\documentclass{amsart}

\usepackage{pst-node}

\usepackage{tikz-cd}

\usepackage{latexsym}
\usepackage{enumerate}
\usepackage{amssymb}
\usepackage{amsfonts}
\usepackage{amsmath}
\usepackage{mathrsfs}
\usepackage{amsthm}
\usepackage{geometry}
\usepackage[colorlinks,
linkcolor=red,
anchorcolor=blue,
citecolor=green
]{hyperref}
\usepackage{cleveref}
\usepackage[all]{xy}

\def\Xint#1{\mathchoice
{\XXint\displaystyle\textstyle{#1}}%
{\XXint\textstyle\scriptstyle{#1}}%
{\XXint\scriptstyle\scriptscriptstyle{#1}}%
{\XXint\scriptscriptstyle\scriptscriptstyle{#1}}%
\!\int}
\def\XXint#1#2#3{{\setbox0=\hbox{$#1{#2#3}{\int}$ }
\vcenter{\hbox{$#2#3$ }}\kern-.6\wd0}}

\def\dashint{\Xint-}

\newcommand{\partialbar}{\bar{\partial}}

\newcommand{\Ric}{\mathrm{Ric}}
\newcommand{\Vol}{\mathrm{Vol}}

\newcommand{\Aut}{\mathrm{Aut}}

\newtheorem{thm}{Theorem}[section]

\newtheorem{prop}{Proposition}[section]

\newtheorem{lmm}{Lemma}[section]

\newtheorem{conj}{Conjecture}[section]
\newtheorem{pbm}{Problem}[section]

\theoremstyle{remark} 
\newtheorem*{rmk}{Remark}

\title{On the \texorpdfstring{$L^2 $}{Lg} volume of Bergman spaces }

\author{Shengxuan Zhou}
\address{Beijing International Center for Mathematical Research\\
Peking University\\
Beijing\\ 100871\\ China}
\email{zhoushx19@pku.edu.cn}

\begin{document}

\begin{abstract}
In this paper, we show that the Calabi volume and Mabuchi volume of Bergman spaces on the product of a projective manifold and a projective space is infinite. Our result is inspired by a conjecture of Shiffman-Zelditch in \cite{shifzel1}.
\end{abstract}

\maketitle

\tableofcontents

\section{Introduction}
\label{intro} 

Let $ M $ be a $n$-dimensional projective manifold, $L$ be a very ample line bundle on $M$, and $\omega_0 \in c_1 (L) $ be a K\"ahler metric on $M$. The pair $(M,L)$ is called a polarized manifold. Then the space of K\"ahler metrics and the space of K\"ahler potentials can be defined as follows:
\begin{equation} \label{definitionspacekahlermetrics}
    \mathcal{K} = \left\lbrace \omega :\;\; \textrm{$d\omega =0 $, and } \omega >0 \right\rbrace ,
\end{equation}
and
\begin{equation} \label{definitionspacekahlerpotentials}
    \mathcal{H}_{[\omega_0 ]} = \left\lbrace f \in C^\infty (M;\mathbb{R}) :\;\; \omega_f = \omega_0 + \sqrt{-1} \partial \partialbar f >0 \right\rbrace .
\end{equation}
By $\partial\partialbar$-lemma, one can see that $\omega' \in [\omega_0] = c_1 (L) $ if and only if $\omega'=\omega_f$ for some $f\in \mathcal{H}_{[\omega_0 ]} $, and such $f$ is unique up to an additive constant. Hence $ \mathcal{K}_{[\omega_0 ]} = \mathcal{K} \cap [\omega_0 ] \cong \mathcal{H}_{[\omega_0 ]} / \mathbb{R} $. For example, see \cite[Chapter VI, Lemma 8.6]{dm1} or \cite[Corollary 2.2]{tgbook1}.

Since $L$ is very ample, for any $k\in\mathbb{N}$, one can construct an embedding $M \to \mathbb{C}P^{N_k } $ by a basis of $H^0 (M,L^k ) $, where $N_k = \dim H^0 (M,L^k ) -1 $. Now we can define the space of Bergman metrics as following.
\begin{eqnarray} \label{definitionbergmanspace}
    \mathcal{B}_{M,L^k} = \left\lbrace \frac{1}{k} \gamma^* \omega_{FS} ,\;\; \gamma : M \to \mathbb{C}P^{N_k } \textrm{ given by a basis of $ H^0 (M,L^k) $ } \right\rbrace ,
\end{eqnarray}
where $\omega_{FS} $ is the Fubin-Study metric of $\mathbb{C}P^{N_k} $. Clearly, $\mathcal{B}_{M,L^k} \subset \mathcal{K}_{[\omega_0 ]} $, $\forall k\in\mathbb{N}$. The space of Bergman metrics will also be abbreviated as $\mathcal{B}_{k}$ when we do not emphasize the manifold $M$ and the line bundle $L$.

The study of the approximation of subspaces $ \mathcal{B}_{M,L^k} $ to space $\mathcal{K}_{[\omega_0 ]}$ is an important topic in K\"ahler geometry. In the seminal work \cite{tg1}, Tian utilized his peak section method to prove that Bergman metrics converge to the original polarized metric in the $C^2$-topology. Consequently, $\mathcal{K}_{[\omega_0 ]} $ can be approximation by $\mathcal{B}_{L,k} $ in the $C^2$-topology. Following this approach, Ruan \cite{wdr1} proved that this convergence holds in the $C^\infty $ sense. Later, Zelditch \cite{sz1}, also Catlin \cite{cat1} independently, used the Szeg\"o kernel to obtain an alternative proof of the $C^\infty $-convergence of Bergman metrics. Furthermore, they provided an asymptotic expansion of Bergman kernel function, which is the K\"ahler potential of the Bergman metric. Such an asymptotic expansion is called the Tian-Yau-Zelditch expansion. This expansion can be also obtained by Tian's peak section method, see Liu-Lu \cite{liuzql1}. By using the heat kernel, Dai-Liu-Ma \cite{dailiuma1} gave another proof of the Tian-Yau-Zelditch expansion (see also Ma-Marinescu’s book \cite{mamari1}). There are many important applications of the Tian–Yau–Zelditch expansion, for example, \cite{don1} and \cite{kwz2}. 

There are two natural metrics on the space $ \mathcal{K}_{[\omega_0 ]} $.

The Calabi metric \cite{calabi1, calabi2} serves as the natural $L^2$ metric on $\mathcal{K}_{[\omega_0 ]}$ defined by
\begin{eqnarray} \label{definitioncalabimetric}
    G_{Ca, \mathcal{K}_{[\omega_0 ]}} (\dot{\varphi} ,\dot{\phi}) = \int_{M} \Delta_\omega \dot{\varphi} \Delta_\omega \dot{\phi} dV_\omega = \int_{M} \left(\mathrm{tr}_{\omega } \sqrt{-1} \partial \partialbar \dot{\varphi} \right) \cdot \left( \mathrm{tr}_{\omega } \sqrt{-1} \partial \partialbar \dot{\phi} \right) dV_\omega ,
\end{eqnarray}
where $\omega \in [\omega_0 ] $, and $\dot{\varphi} ,\dot{\phi} \in T_\omega \mathcal{K}_{[\omega_0 ]} \cong C^\infty (M,\mathbb{R}) /\mathbb{R} $. It is known that the sectional curvatures of the Calabi metric $G_{Ca, \mathcal{K}_{[\omega_0 ]}} $ on $ \mathcal{K}_{[\omega_0 ]} $ are all equal to $1 $, and $ \mathcal{K}_{[\omega_0 ]} $ does not have conjugate points with respect to the Calabi metric \cite{calamai1}. Thus, the infinite-dimensional Riemannian manifold $( \mathcal{K}_{[\omega_0 ]} , G_{Ca, \mathcal{K}_{[\omega_0 ]}} )$ the space is a portion of an infinite dimensional sphere of constant curvature $1$.

Another natural $L^2$ metric on $\mathcal{K}_{[\omega_0 ]}$ is the Mabuchi metric arised in \cite{mabuchi1}. At first, one can defined an $L^2$ metric on $\mathcal{H}_{[\omega_0 ]}$ by
\begin{eqnarray} \label{definitionmabuchimetric1}
    {G}_{Ma, \mathcal{H}_{[\omega_0 ]}} (\dot{\varphi} ,\dot{\phi}) = \int_{M} \dot{\varphi} \dot{\phi} dV_{\omega_f} ,
\end{eqnarray}
where $f \in \mathcal{H}_{[\omega_0 ]} $, $\omega_f = \omega_0 + \sqrt{-1} \partial \partialbar f >0 $, and $\dot{\varphi} ,\dot{\phi} \in T_f \mathcal{H}_{[\omega_0 ]} \cong C^\infty (M,\mathbb{R}) $. 

Clearly, ${G}_{Ma, \mathcal{H}_{[\omega_0 ]}}$ is invariant under the action $A_c : \mathcal{H}_{[\omega_0 ]} \to \mathcal{H}_{[\omega_0 ]} $, $A_c (f) = f+c$, $\forall c\in \mathbb{R}$. Hence (\ref{definitionmabuchimetric1}) gives an $L^2$ metric ${G}_{Ma, \mathcal{K}_{[\omega_0 ]}}$ on $ \mathcal{K}_{[\omega_0 ]} \cong \mathcal{H}_{[\omega_0 ]} / \mathbb{R} $ such that the quotient map $ ( \mathcal{H}_{[\omega_0 ]} , {G}_{Ma, \mathcal{H}_{[\omega_0 ]}} ) \to ( \mathcal{K}_{[\omega_0 ]} , {G}_{Ma, \mathcal{K}_{[\omega_0 ]}} ) $ becomes a Riemannian submersion. By definition, we have
\begin{eqnarray} \label{definitionmabuchimetric2}
    {G}_{Ma, \mathcal{K}_{[\omega_0 ]}} (\dot{\varphi} ,\dot{\phi}) = \int_{M} \left( \dot{\varphi} - \dashint_M \dot{\varphi} dV_{\omega } \right) \left( \dot{\phi} - \dashint_M \dot{\phi} dV_{\omega } \right) dV_{\omega } ,
\end{eqnarray}
where $\omega \in \mathcal{K}_{[\omega_0 ]} $, $\dot{\varphi} ,\dot{\phi} \in T_\omega \mathcal{K}_{[\omega_0 ]} \cong C^\infty (M,\mathbb{R}) /\mathbb{R} $, and $\dashint_M \dot{\phi} dV_{\omega } = \frac{1}{\Vol_\omega (M)} \int_M \dot{\phi} dV_{\omega } $.

Since $[ \omega_0 ] = c_1 (L)\in H^2 ( M ; \mathbb{Z} ) $, we have $\mathcal{B}_{M,L^k} \subset \mathcal{K}_{ [ \omega_0 ] } $. Hence one can obtain a Riemannian metric on $\mathcal{B}_{M,L^k} $ by the restriction of the Calabi metric or the Mabuchi metric on $\mathcal{B}_{M,L^k}$. Let ${G}_{Ca, M,L^k}$ and ${G}_{Ma, M,L^k}$ denote the restriction of the Calabi metric and the Mabuchi metric on $\mathcal{B}_{M,L^k}$, respectively. Then ${G}_{Ca, M,L^k}$ and ${G}_{Ma, M,L^k}$ are invariant under the action of $\Aut (M)$ in some sense. In fact, we have the following property.

\begin{prop}\label{propactioninv}
Let $ M $ be a $n$-dimensional projective manifold, $L$ be a very ample line bundle on $M$, and $\gamma\in\Aut (M)$. Then the pullback 
\begin{eqnarray} \label{pullbackbergmanspaces}
    \gamma^* : \mathcal{B}_{M,L^k} \to \mathcal{B}_{M,\gamma^* L^k}
\end{eqnarray}
satisfies $\gamma^* {G}_{Ca, M,L^k} = {G}_{Ca, M,\gamma^* L^k} $ and $\gamma^* {G}_{Ma, M,L^k} = {G}_{Ma, M,\gamma^* L^k} $.
\end{prop}

Let $\mu_{Ca} $ and $\mu_{Ma} $ be the measures on $\mathcal{B}_{M,L^k}$ corresponding to the Riemannian metrics $ {G}_{Ca, M,L^k} $ and $ {G}_{Ma, M,L^k} $, respectively. Combining Proposition \ref{propactioninv} with the parameterization of $\mathcal{B}_{M,L^k}$ (see \cite[Section 5]{shifzel1} or Section \ref{matrix}), we can obtain the following result.

\begin{thm}
\label{thmexampleCPN}
Let $M $ be a projective manifold, and $L $ be a very ample line bundle on $M$. Suppose $M=M_1 \times \mathbb{C}P^n $, where $M_1$ is a projective manifold (which can be a single point) and $n\in\mathbb{N}$. Then for any $k\in\mathbb{N}$, $\mu_{Ca} ( \mathcal{B}_{M, L^k} ) = \mu_{Ma} ( \mathcal{B}_{M, L^k} ) = \infty $.
\end{thm}

\begin{rmk}
    Actually, Theorem \ref{thmexampleCPN} holds for any polarized manifold $(M,L)$ such that the automorphism group $\Aut (M,L)$ is non-compact.
\end{rmk}

Our Theorem \ref{thmexampleCPN} is inspired by the following conjecture of Shiffman-Zelditch in \cite{shifzel1}.

 \begin{conj}[{\cite[Conjecture 5.1]{shifzel1}}]
\label{shifzelconj}
 The Calabi volume $\mu_{Ca} ( \mathcal{B}_{M, L^k} ) $ is finite for each $k$.
\end{conj}

As highlighted by Shiffman-Zelditch in \cite{shifzel1}, if this conjecture is true, one can obtain a rigorous definition of the Polyakov path integral over metrics, which used the Calabi metric to define its volume form. Such a development would undoubtedly have an impact on both K\"ahler geometry and mathematical physics, fostering new connections between the two fields. 

Since Theorem \ref{thmexampleCPN} shows that Shiffman-Zelditch's original conjecture may not always hold, we pose the following problem as a modification of Shiffman-Zelditch's original conjecture.

\begin{pbm}[Modified Shiffman-Zelditch's conjecture]
\label{shifzelpbm}
Let $(M,L)$ be a polarized manifold, and $\omega_0 \in c_1 (L) $ be a K\"ahler metric on $M$. Assume that $\Aut (M,L) $ is compact.
\begin{enumerate}[(a).]
    \item Is there an integer $k_0 \in\mathbb{N} $ such that for any $k\geq k_0$, the Calabi volume $\mu_{Ca} ( \mathcal{B}_{M, L^k} ) $ and the Mabuchi volume $\mu_{Ma} ( \mathcal{B}_{M, L^k} ) $ are finite?
    \item Is it possible to obtain a rigorous definition of the Polyakov path integral over metrics by using the Calabi metric or the Mabuchi metric to define its volume form?
\end{enumerate}
\end{pbm}

This paper is organized as follows. In Section \ref{matrix}, we recall the parameterization of $\mathcal{B}_{L,k}$ and use it to describe the Calabi volume form on $\mathcal{B}_{L,k} $. Then we will prove Proposition \ref{propactioninv} and Theorem \ref{thmexampleCPN} in Section \ref{proofmainthm}.

\textbf{Acknowledgement.} The author wants to express his deep gratitude to Professor Gang Tian for constant encouragement.

\section{The parameterization of \texorpdfstring{$\mathcal{B}_{M,L^k}$}{Lg}}
\label{matrix}

In this section, we recall the parameterization of $\mathcal{B}_{M,L^k}$. The following parameterization is known (for example, see \cite[Section 5]{shifzel1}), but we include the details for convenience.

Let $\{ S_i \}_{i=0}^{N_k} $ be a basis of $ H^0 (M,L^k) $, and let $h $ be a Hermitian metric on $L$ such that the Ricci curvature $\omega_h = \Ric (h) >0 $. For any matrix $A = (a_{i,j})_{0\leq i,j \leq N_k} \in GL(N_k +1 ,\mathbb{C} ) $, one can construct a K\"ahler metric
\begin{eqnarray}
    \omega_A = \frac{1}{k} F_A^* \omega_{FS} = \frac{1}{2\pi } \omega_h + \frac{\sqrt{-1}}{2k\pi } \partial \partialbar \log \left( \sum_{i=0}^{N_k} \left\Vert \sum_{j=0}^{N_k} a_{i,j} S_j \right\Vert_{h^k}^2 \right) \in \mathcal{B}_{M,L^k} ,
\end{eqnarray}
where $\omega_{FS}$ is the Fubini-Study metric on $ \mathbb{C}P^{N_k}$, and $F_A$ is the embedding $M\to \mathbb{C}P^{N_k}$ given by the basis $\big\{  \sum_{j=0}^{N_k} a_{i,j} S_j \big\}_{i=0}^{N_k}$. It is easy to see that $\omega_A$ is independent of the choice of $h$, and a K\"ahler metric $\omega \in \mathcal{B}_{M,L^k} $ if and only if there exists a matrix $A = (a_{i,j})_{0\leq i,j \leq N_k} \in GL(N_k +1 ,\mathbb{C} ) $ such that $\omega = \omega_A $.

Let $B = (b_{i,j})_{0\leq i,j \leq N_k} \in GL(N_k +1 ,\mathbb{C} ) $. Assume that $\omega_A =\omega_B $. Then 
$$ \frac{\sqrt{-1}}{2\pi } \partial \partialbar \log \left( \sum_{i=0}^{N_k} \left\Vert \sum_{j=0}^{N_k} a_{i,j} S_j \right\Vert_{h^k}^2 \right) = \frac{\sqrt{-1}}{2\pi } \partial \partialbar \log \left( \sum_{i=0}^{N_k} \left\Vert \sum_{j=0}^{N_k} b_{i,j} S_j \right\Vert_{h^k}^2 \right) ,$$
and hence there exists a constant $\rho >0$ such that
\begin{eqnarray} \label{bergmanequalrho}
    \sum_{i=0}^{N_k} \left\Vert \sum_{j=0}^{N_k} a_{i,j} S_j \right\Vert_{h^k}^2 = \rho^2 \sum_{i=0}^{N_k} \left\Vert \sum_{j=0}^{N_k} b_{i,j} S_j \right\Vert_{h^k}^2 .
\end{eqnarray}
Let $U_x$ be an open neighborhood of $x\in M$, and let $e_x$ be a local frame of $L$ on $U_x$. Without loss of generality, we can assume that there exists a biholomorphic map $ F_x : U_x \to B_1 (0) \subset \mathbb{C}^n $. Write $S_i = f_i e_x^k $. Define 
\begin{eqnarray} \label{preAAtBBtequation}
    \varrho (z,w) = \sum_{i,j,k=0}^{N_k} (a_{i,j} \bar{a}_{i,k} - \rho^2 b_{i,j} \bar{b}_{i,k} ) f_j (F_x^{-1}(z)) \bar{ f}_k (F_x^{-1}(\bar{w} )) \in \mathcal{O} ( B_1 (0) ).
\end{eqnarray}
Then (\ref{bergmanequalrho}) implies that $\varrho (z,\bar{z}) = 0$, $\forall z\in B_1 (0) $. Hence $\varrho \in\mathcal{O} ( B_1 (0) ) $ implies that $\varrho =0$. Since $\{ S_i \}_{i=0}^{N_k} $ is a basis, it can be observed that $\{ f_i \}_{i=0}^{N_k} $ is linearly independent. Now we can obtain from (\ref{preAAtBBtequation}) that $A^t \bar{A} = \rho^2 B^t \bar{B} $, and hence $Q=\rho^{-1}AB^{-1} \in U (N_k +1) $, where $A^t$ is the transpose of $A$, and $\bar{A}$ is the conjugate of $A$. 

If $A=\rho Q B$ for some $\rho >0$ and $Q\in U (N_k +1) $, then one can easily to see that $\omega_A =\omega_B$. It follows that the corresponding $A\mapsto \omega_A $ gives a diffeomorphism $(\mathbb{C}^* \times SU (N_k +1) )\backslash GL (N_k +1 ,\mathbb{C}) \to \mathcal{B}_{M,L^k} $. Thus, $\mathcal{B}_{M,L^k}$ can be parametrized by $(\mathbb{C}^* \times SU (N_k +1) )\backslash GL (N_k +1 ,\mathbb{C}) = SU (N_k +1) \backslash SL (N_k +1 ,\mathbb{C}) $.

\begin{prop}[{\cite[Section 5]{shifzel1}}] \label{propSUSLmapBMLK}
    The map
    \begin{eqnarray} \label{formulaSUSLmapBMLK}
 \Psi : \left\{ \begin{array}{ll}
SU (N_k +1) \backslash SL (N_k +1 ,\mathbb{C}) & \to \mathcal{B}_{M,L^k} ,\\
\;\;\;\;\;\;\;\;\;\;\;\;\;\;\;\;\;\;\;\; [A] & \mapsto \omega_A
\end{array} \right.
\end{eqnarray}
is well-defined, where $A\in SL (N_k +1 ,\mathbb{C}) $ is a representative of $[A]\in SU (N_k +1) \backslash SL (N_k +1 ,\mathbb{C}) $. Moreover, $\Psi$ is a diffeomorphism.
\end{prop}

Now we can parameterize $\mathcal{B}_{M,L^k}$ by $SU (N_k +1) \backslash SL (N_k +1 ,\mathbb{C})$. However, calculating the integral over $SU (N_k +1) \backslash SL (N_k +1 ,\mathbb{C})$ is a bit complicated. Our next goal is to construct a global coordinate of $SU (N_k +1) \backslash SL (N_k +1 ,\mathbb{C})$.

Let $A\in SL (N_k +1 ,\mathbb{C}) $. By the Gram-Schmidt process, one can find $Q\in SU (N_k +1) $ such that $R= (r_{i,j})_{0\leq i,j \leq N_k} =Q^{-1} A $ is an upper triangular matrix such that $r_{i,i} \in (0,\infty ) $ and $\prod_{i=0}^{N_k } r_{i,i} = 1 $. This decomposition is the complex version of the famous QR decomposition \cite[Theorem 2.1.14]{hojo1}. Now we assume that $\tilde{Q}\in SU (N_k +1 )$ and $\tilde{R}= (\tilde{r}_{i,j})_{0\leq i,j \leq N_k} $ is an upper triangular matrix such that $\tilde{r}_{i,i} \in (0,\infty ) $, $\prod_{i=0}^{N_k } \tilde{r}_{i,i} = 1 $ and $\tilde{Q}\tilde{R} = A$. Then we can conclude that $Q^{-1} \tilde{Q} = R\tilde{R}^{-1} \in SU(N_k +1) $ is an upper triangular matrix such that the diagonal elements are positive numbers. It follows that $ Q^{-1} \tilde{Q} = R\tilde{R}^{-1} = I_{N_k +1} $. Hence for each $A\in SL (N_k +1 ,\mathbb{C}) $, the pair $(Q,R)$ is uniqueness.

Let $\mathcal{T}_{N_k} $ be the following submanifold of $SL (N_k +1 ,\mathbb{C})$:
\begin{eqnarray}
    \mathcal{T}_{N_k} = \left\lbrace R \in SL (N_k +1 ,\mathbb{C}) : \textrm{ $r_{i,i } 
\in \mathbb{R}_+ $, $\prod_{i=0}^{N_k} r_{i,i} =1 $, and $r_{i,j} =0$, $\forall i>j$ } \right\rbrace .
\end{eqnarray}
By the QR decomposition as above, one can see that $\mathcal{T}_{N_k}$ is a complete set of representatives of $ SU (N_k +1) \backslash SL (N_k +1 ,\mathbb{C}) $, and the inclusion map $\mathcal{T}_{N_k} \to SL (N_k +1 ,\mathbb{C}) $ gives a diffeomorphism:
\begin{eqnarray} \label{finalparametrimap}
    \Phi : \left\{ \begin{array}{ll}
\;\;\;\;\;\;\;\;\;\;\;\;\; \mathcal{T}_{N_k}  & \to SU (N_k +1) \backslash SL (N_k +1 ,\mathbb{C}) \stackrel{\Psi}{\cong} \mathcal{B}_{L,k} ,\\
 R = (r_{i,j})_{0\leq i,j \leq N_k} & \mapsto \;\;\;\;\;\;\;\;\;\;\;\;\;\;\;\;\;\; [R] \;\;\;\;\;\;\;\;\;\;\;\;\;\;\;\;\;\;\;\; \leftrightarrow \omega_R .
\end{array} \right.
\end{eqnarray}

Moreover, the elements $r_{i,j}$ give a diffeomorphism 
\begin{eqnarray} \label{finalparametrieuclidean}
    \Upsilon : \mathbb{R}_{+}^{N_k} \times \mathbb{C}^{\frac{N_k (N_k +1) }{2}} \to SU (N_k +1) \backslash SL (N_k +1 ,\mathbb{C}) \stackrel{\Psi}{\cong} \mathcal{B}_{L,k} .
\end{eqnarray}
Then for any open subset $\Omega\subset \mathcal{B}_{M,L^k} $, the Calabi volume or Mabuchi volume of $\Omega$ can be expressed as
\begin{eqnarray} \label{calabivolumematrix}
    \mu_{G} ( \Omega ) = \int_{\Upsilon^{-1} ( \Phi^{-1} ( \Psi^{-1} (\Omega ) ) ) } \sqrt{\det (\Upsilon^* \Phi^* \Psi^* G) } ,
\end{eqnarray}
and hence can be calculated on the Euclidean domain $\Upsilon^{-1} ( \Phi^{-1} ( \Psi^{-1} (\Omega ) ) ) \subset \mathbb{R}_{+}^{N_k} \times \mathbb{C}^{\frac{N_k (N_k +1) }{2}} $, where $G$ is the Calabi metric or Mabuchi metric, respectively. 

\section{Proof of Theorem {\ref{thmexampleCPN}}}
\label{proofmainthm}

In this section, we prove Theorem {\ref{thmexampleCPN}}.

At first, we prove ${G}_{Ca, M,L^k}$ and ${G}_{Ma, M,L^k}$ are invariant under the action of $\Aut (M)$.

\begin{prop}[= Proposition \ref{propactioninv}]
Let $ M $ be a $n$-dimensional projective manifold, $L$ be a very ample line bundle on $M$, and $\gamma\in\Aut (M)$. Then the pullback $\gamma^* $ gives a diffeomorphism
\begin{eqnarray}
    \gamma^* : \mathcal{B}_{M,L^k} \to \mathcal{B}_{M,\gamma^* L^k} .
\end{eqnarray}
Moreover, we have $\gamma^* {G}_{Ca, M,L^k} = {G}_{Ca, M,\gamma^* L^k} $ and $\gamma^* {G}_{Ma, M,L^k} = {G}_{Ma, M,\gamma^* L^k} $.
\end{prop}

\begin{proof}
Let $\omega\in \mathcal{B}_{M,L^k}$. Then we have an embedding $F_\omega : M\to\mathbb{C}P^{N_k} $ such that $F_\omega^* \mathcal{O} (1) = L^k $ and $F_\omega^* \omega_{FS} = k\omega $, where $\omega_{FS}$ is the Fubini-Study metric on $\mathbb{C}P^{N_k} $. Hence $k \gamma^* \omega = \gamma^* F_\omega^* \omega_{FS} = (F_\omega \circ \gamma )^* \omega_{FS} $ and $ (F_\omega \circ \gamma )^* \mathcal{O}(1) = \gamma^* L^k $. It follows that $\gamma^* \omega \in \mathcal{B}_{M,\gamma^* L^k} $. Clearly, the map $\gamma^* : \mathcal{B}_{M,L^k} \to \mathcal{B}_{M,\gamma^* L^k} $ is smooth. Similarly, $(\gamma^{-1} )^* $ is also smooth, and it is the inverse map of $\gamma^*$. Now we can conclude that $\gamma^* : \mathcal{B}_{M,L^k} \to \mathcal{B}_{M,\gamma^* L^k} $ is a diffeomorphism.

By definition, for any $ \dot{\varphi} ,\dot{\phi} \in T_\omega \mathcal{B}_{M,L^k} \leq T_\omega \mathcal{K}_{[\omega_0]} \cong C^\infty (M,\mathbb{R}) /\mathbb{R} $, we have $d\gamma^*_\omega \dot{\varphi} = \dot{\varphi} \circ \gamma $. Hence 
\begin{eqnarray*}
    \left(\gamma^* {G}_{Ca, M,L^k} \right)_{\gamma^* \omega} ( d\gamma^*_\omega \dot{\varphi} ,d\gamma^*_\omega \dot{\varphi} ) & = & \left( {G}_{Ca, M,L^k} \right)_{\omega } ( \dot{\varphi} , \dot{\phi} ) \\
    & = & \int_{M} \Delta_\omega \dot{\varphi} \Delta_\omega \dot{\phi} dV_\omega \\
    & = & \int_{M} \left( \Delta_\omega \dot{\varphi} \Delta_\omega \dot{\phi} \right) \circ \gamma dV_{\gamma^*\omega} \\
    & = & \int_{M} \left( \Delta_{\gamma^*\omega} ( \dot{\varphi} \circ \gamma ) \right) \cdot \left( \Delta_{\gamma^*\omega} ( \dot{\phi} \circ \gamma ) \right)  dV_{\gamma^*\omega} \\
    & = & \left( {G}_{Ca, M,\gamma^* L^k} \right)_{\gamma^* \omega } ( d\gamma^*_\omega \dot{\varphi} ,d\gamma^*_\omega \dot{\varphi} ) ,
\end{eqnarray*}
where $ \dot{\varphi} ,\dot{\phi} \in T_\omega \mathcal{B}_{M,L^k} $. It follows that $\gamma^* {G}_{Ca, M,L^k} = {G}_{Ca, M,\gamma^* L^k} $. By a similar argument, one can obtain $\gamma^* {G}_{Ma, M,L^k} = {G}_{Ma, M,\gamma^* L^k} $.
\end{proof}

Before proving Theorem {\ref{thmexampleCPN}}, we need the following lemma.

\begin{lmm}
    \label{counterequivariantlmm}
Let $ M $ be an $n$-dimensional projective manifold, and let $L$ be a very ample line bundle on $M$. Assume that there exist a subgroup $\Gamma \leq \mathrm{Aut} (M) $ and $\omega \in \mathcal{B}_{M,L^k} $ such that $\gamma^* L\cong L $, $\forall \gamma\in\Gamma $, and $\Gamma \omega =\{ \gamma^* \omega :\;\; \gamma\in\Gamma \} $ is an infinite discrete subset of $ \mathcal{B}_{M,L^k} $. Then the Calabi volume $\mu_{Ca} ( \mathcal{B}_{M,L^k} ) $ and the Mabuchi volume $\mu_{Ma} ( \mathcal{B}_{M,L^k} )$ are infinite for each $k$.
\end{lmm}

\begin{proof}
    Since ${G}_{Ca, M,L^k} $ is a smooth Riemannian metric on $\mathcal{B}_{M,L^k} $, one can find a constant $\epsilon >0 $ such that the metric ball $B_{10\epsilon } (\omega ) $ is precompact in $\mathcal{B}_{M,L^k} $. By the discreteness of $\Gamma \omega $, we can assume that $B_{10\epsilon } (\omega ) \cap \Gamma \omega = \emptyset $. Note that for any $\gamma\in\Gamma $, $\gamma^* : \mathcal{B}_{M,L^k} \to \mathcal{B}_{M,L^k} $ is isometric. Hence
    $$ \mu_{Ca} ( \mathcal{B}_{M,L^k} ) \geq \mu_{Ca} \left( \cup_{\gamma\in\Gamma} B_\epsilon (\gamma^* \omega ) \right) = \sum_{\gamma^* \omega } \mu_{Ca} \left( B_\epsilon ( \omega ) \right) = \infty .$$
    By a similar argument, we can conclude that $\mu_{Ma} ( \mathcal{B}_{M,L^k} ) =\infty  $, and the proof is complete.
\end{proof}

We are ready to prove Theorem {\ref{thmexampleCPN}} now. Let us begin with the special case $M_1 = *$.

\begin{prop} \label{propCPNexample}
 Let $ M = \mathbb{C}P^{n } $, and $L = \mathcal{O} (1) $. Then for any $k\in\mathbb{N}$ and $\omega \in \mathcal{B}_{M,L^k} ) $, there exists a subgroup $\Gamma \leq \Aut (M) \cong GL (n+1 ,\mathbb{C}) $ such that $\Gamma \omega =\{ \gamma^* \omega : \gamma\in\Gamma \} $ is an infinite discrete subset of $ \mathcal{B}_{M,L^k} $. Moreover, for each $k\in\mathbb{N}$, $\mu_{Ca} ( \mathcal{B}_{M,L^k} ) = \mu_{Ma} ( \mathcal{B}_{M,L^k} ) =\infty $.
\end{prop}

\begin{proof}
 Let $( Z_0 : Z_1 : \cdots : Z_n )$ be the homogeneous coordinate of $M$. Then $\{ Z^{k_0}_0 \cdots Z_n^{k_n} \}_{k_0 +\cdots +k_n =k }$ gives a basis of $H^0 (M,L^k )$, and hence $ N_k +1 = \dim H^0 (M,L^k ) = \frac{(n+k )!}{n! k!} $. 
 
 By induction on $d\in\mathbb{N}$, one can construct a lexicographic order on $\mathbb{Z}_{\geq 0}^d$ such that $(p_1 ,\cdots ,p_d) > (q_1 ,\cdots ,q_d) $ if and only if:
 \begin{itemize}
     \item $p_d > q_d$, or
     \item $d>1$, $p_d =q_d $, and $(p_1 ,\cdots ,p_{d-1}) > (q_1 ,\cdots ,q_{d-1}) $.
 \end{itemize}

Let $S_i = Z^{k_{0,i}}_0 \cdots Z_n^{k_{n,i}} $, $i=0,1,\cdots ,N_k $, be an arrangement of the basis $\{ Z^{k_0}_0 \cdots Z_n^{k_n} \}_{k_0 +\cdots +k_n =k }$ of $H^0 (M,L^k )$ such that $i\geq j$ if and only if $(k_{0,i} ,\cdots ,k_{n,i}) \geq (k_{0,j} ,\cdots ,k_{n,j}) $.

Hence we have $S_{N_{k}} = Z_n^k $ and $ S_{N_{k} -1} = Z_{n-1} Z_{n}^{k-1} $. Let $R= (r_{i,j})_{0\leq i,j \leq N_k} = \Phi^{-1} ( \Psi^{-1} (\omega) ) $ be the upper triangular matrix in (\ref{finalparametrimap}), and
\begin{eqnarray}
    \Gamma = \left\lbrace \gamma_m = I_{n+1} + m E_{n-1,n} :\; m\in\mathbb{Z} \right\rbrace \leq GL (n+1 ,\mathbb{C}) \cong \Aut (M) ,
\end{eqnarray}
where $I_{n+1} =  (\delta_{i,j})_{0\leq i,j \leq n} $, $E_{n-1,n} = (\delta_{i,n-1} \delta_{j,n } )_{0\leq i,j \leq n} $, and $\delta_{k,l}$ is the Kronecker symbol.

By definition, we have $\gamma^*_m (Z_{i}) = Z_i - m \delta_{i,n-1} Z_n $, and hence for $0\leq i < j\leq N_k $, there exists $a_{i,j,m} \in\mathbb{C} $ such that $\gamma^*_m (S_{i}) = S_i + \sum\limits_{0\leq i < j\leq N_k} a_{i,j,m} S_j $. In particular, $a_{N_{k} -1 ,N_k} = -m $. Write $A_m = (\delta_{i,j} + a_{i,j,m} )_{0\leq i,j \leq N_k} $. Fix a Hermitian metric $h$ on $L$ such that $\omega_h = \Ric (h) >0 $. Clearly, $ A_m =A^m_1$. Since $R= (r_{i,j})_{0\leq i,j \leq N_k} = \Phi^{-1} ( \Psi^{-1} (\omega) )$, we have
\begin{eqnarray*}
    \gamma_m^* \omega & = & \gamma_m^* \Psi (\Phi (R) ) \\
    & = & \gamma_m^* \left( \frac{1}{2\pi} \omega_h + \frac{\sqrt{-1}}{2k\pi}  \partial\partialbar \log \left( \sum_{i=0}^{N_k} \left\| \sum_{j=i}^{N_k} r_{i,j} S_j \right\|_{h^k}^2 \right) \right) \\
    & = & \frac{1}{2\pi} \gamma_m^*\omega_h + \frac{\sqrt{-1}}{2k\pi} \partial\partialbar \log \left( \sum_{i=0}^{N_k} \left\| \sum_{j=i}^{N_k} r_{i,j} \gamma_m^* S_j \right\|_{\gamma_m^* h^k}^2 \right) \\
    & = & \frac{1}{2\pi} \omega_h + \frac{\sqrt{-1}}{2k\pi} \partial\partialbar \log \left( \sum_{i=0}^{N_k} \left\| \sum_{j=i}^{N_k} \sum_{l=j}^{N_k} r_{i,j} (\delta_{j,l} + a_{j,l,m} ) S_l \right\|_{ h^k}^2 \right) \\
    & = & \Psi ( \Phi ( RA_m) ) .
\end{eqnarray*}
Note that $\omega_A$ is independent on the choice of $h$.

Since for any $m\in\mathbb{Z}$, $A_m$ is an upper triangular matrix with all elements on the diagonal equal to $1$, we observe that $RA_m\in SL(N_k +1 ,\mathbb{C})$ is also an upper triangular matrix, and the elements on the diagonal of $R-RA_m$ are all equal to $0$. Hence the elements on the diagonal of $ RA_m \in SL(N_k +1 ,\mathbb{C}) $ are all positive. By the argument in Section \ref{matrix}, $\{RA_m  \}_{m\in\mathbb{Z}}$ derives a discrete subset of $SU (N_k +1) \backslash SL(N_k +1 ,\mathbb{C})$ if and only if $\{RA_m  \}_{m\in\mathbb{Z}}$ is a discrete subset of $SL(N_k +1 ,\mathbb{C})$.

It is sufficient to show that $\{RA_m  \}_{m\in\mathbb{Z}}$ is a discrete subset of $SL(N_k +1 ,\mathbb{C})$ now. Write $RA_m = (r'_{i,j,m})_{0\leq i,j \leq N_k} $. Since $a_{N_{k} -1 ,N_k} = -m $, we have $r'_{N_{k} -1 ,N_k,m} = r_{N_{k} -1 ,N_k} - mr_{N_{k} -1 ,N_k -1} $. Then $ r_{N_{k} -1 ,N_k -1} >0$ implies that for any $m_0 \in\mathbb{Z} $, there exists an open neighborhood 
$$U_{m_0} = \left\{ B= (b_{i,j})_{0\leq i,j \leq N_k}:\;\; \left| b_{N_{k} -1 ,N_k } - r_{N_{k} -1 ,N_k } + m_0 r_{N_{k} -1 ,N_k -1 } \right| \leq \frac{r_{N_{k} -1 ,N_k -1 }}{2} \right\} $$ 
of $R$ such that $U_{m_0} \cap \{RA_m  \}_{m\neq m_0} = \emptyset $. 

It follows that $\{RA_m  \}_{m\in\mathbb{Z}}$ is a discrete subset of $SL(N_k +1 ,\mathbb{C})$, and hence $\Gamma \omega =\{ \gamma^* \omega : \gamma\in\Gamma \} $ is an infinite discrete subset of $ \mathcal{B}_{M,L^k} $. By Lemma \ref{counterequivariantlmm}, one can conclude that the Calabi volume $\mu_{Ca} ( \mathcal{B}_{M,L^k} ) $ and the Mabuchi volume $\mu_{Ma} ( \mathcal{B}_{M,L^k} ) $ are infinite for each $k$. This is our assertion.
\end{proof}

We now consider Theorem \ref{thmexampleCPN} in the case $L=\pi_1^* L_1 \otimes \pi_2^* \mathcal{O} (k_0) $ with $k_0\in\mathbb{N}$, where $\pi_1 : M\to M_1$ and $\pi_2 :M\to \mathbb{C}P^n $ are the projections.

\begin{thm} \label{thmexampleCPNproductL}
Let $M_1$ be a projective manifold (which can be a single point), and $L_1$ be a very ample line bundle on $M_1$. Consider $M=M_1 \times \mathbb{C}P^n $ for some $n\in\mathbb{N}$, and $L=\pi_1^* L_1 \otimes \pi_2^* \mathcal{O} (k_0 ) $ with $k_0 \in\mathbb{N} $, where $\pi_1 : M\to M_1$ and $\pi_2 :M\to \mathbb{C}P^n $ are the projections. Then for each $k$, $\mu_{Ca} ( \mathcal{B}_{M,L^k} ) = \mu_{Ma} ( \mathcal{B}_{M,L^k} ) =\infty $.
\end{thm}

\begin{proof}
Let $h_1$ and $h_2$ be Hermitian metrics on $L_1$ and $L_2 = \mathcal{O} (k_0 )$, respectively. Assume that $\omega_{h_1} = \Ric (h_1) $ and $\omega_{h_2} = \Ric (h_2) $ are K\"ahler metrics on $M_1$ and $ \mathbb{C}P^n $, respectively. Then $h=\pi_1^* h_1 \otimes \pi_2^* h_2$ is a Hermitian metrics on $L$ such that $\omega_h = \Ric (h) = \pi_1^* \omega_{h_1} + \pi_2^* \omega_{h_2} $ is a K\"ahler metrics on $M $. Fix $x\in M_1$, and let $ \mathfrak{F}_x $ denote the natural embedding $\mathbb{C}P^{n} \cong \{ x \} \times \mathbb{C}P^{n} \to M $. Clearly, $ \mathfrak{F}_x^* (L,h ) \cong (\mathcal{O} (1) , h_2 ) $, and $ \mathfrak{F}^*_x \omega = \Ric (h_2) = \omega_{h_2} $.

Let $\{ S_i \}_{i=0}^{N_k} $ be a basis of $ H^0 (M,L^k) $, and let $\omega $ denotes the Bergman metric
$$\omega = \frac{1}{2\pi} \omega_h + \frac{\sqrt{-1}}{2k\pi}  \partial\partialbar \log \left( \sum_{i=0}^{N_k} \left\| S_i \right\|_{h^k}^2 \right) \in \mathcal{B}_{M,L^k} .$$

Let $\{ S'_i \}_{i=0}^{N'_k} $ be a basis of $ H^0 (\mathbb{C}P^n,\mathcal{O} (k k_0 )) $. Then we can find constants $a_{i,j } \in\mathbb{C} $ such that $\mathfrak{F}_x^* S_i = \sum_{j=0}^{N'_k} a_{i,j } S'_j $, $0\leq i\leq N_k$, $0\leq j\leq N'_k$. Hence we can use the Gram-Schmidt process to show that for any $0\leq j,l\leq N'_k$, there exist $b_{j,l} \in\mathbb{C} $ such that 
$$ \left( \sum_{i=1}^{N_k} \Vert S_i \Vert_{h^k}^2 \right) \circ \mathfrak{F}_x = \sum_{j=0}^{N'_k} \left\Vert \sum_{l=0}^{N'_k} b_{j,l } S'_l \right\Vert_{h_2^k}^2 .$$
It follows that
\begin{eqnarray*}
    \mathfrak{F}^*_x \omega & = & \mathfrak{F}^*_x \left( \frac{1}{2\pi} \omega_h + \frac{\sqrt{-1}}{2k\pi}  \partial\partialbar \log \left( \sum_{i=0}^{N_k} \left\| S_i \right\|_{h^k}^2 \right) \right) \\
    & = & \frac{1}{2\pi} \omega_{h_2} + \frac{\sqrt{-1}}{2k\pi}  \partial\partialbar \log \left( \left( \sum_{i=0}^{N_k} \left\| S_i \right\|_{h^k}^2 \right) \circ \mathfrak{F}_x \right) \\
    & = & \frac{1}{2\pi} \omega_{h_2} + \frac{\sqrt{-1}}{2k\pi}  \partial\partialbar \log \left( \sum_{j=0}^{N'_k} \left\Vert \sum_{l=0}^{N'_k} b_{j,l } S'_l \right\Vert_{h_2^k}^2 \right) \in \mathcal{B}_{\mathcal{C}P^{n }, \mathcal{O}( k_0)^k} = \mathcal{B}_{\mathcal{C}P^{n }, \mathcal{O}( kk_0 )} .
\end{eqnarray*}

Let $\Gamma $ be the subgroup of $\Aut (\mathbb{C}P^n ) \cong GL (n+1 ,\mathbb{C})$ in Proposition \ref{propCPNexample} satisfying that $\Gamma (\mathfrak{F}^*_x \omega ) $ is an infinite discrete subset of $\mathcal{B}_{\mathcal{C}P^{n }, \mathcal{O}( kk_0 )} $. Now we consider the group 
$$ \Gamma_M = \mathrm{id}_{M_1} \times \Gamma = \{ (\mathrm{id}_{X_1} ,\gamma ) :\; \; \gamma\in\Gamma \} \leq \Aut (X ) .$$
It is easy to see that $\gamma \mapsto (\mathrm{id}_{M_1} , \gamma ) $ gives a natural isomorphism $\Gamma \to \Gamma_M $. For any $\gamma\in\Gamma $, we have
\begin{eqnarray*}
    \mathfrak{F}^*_x (\mathrm{id}_{M_1} , \gamma )^* \omega & = & \gamma^* \mathfrak{F}^*_x \omega ,
\end{eqnarray*}
and hence $\mathfrak{F}^*_x \left(\Gamma_M \omega \right) = \Gamma (\mathfrak{F}^*_x \omega ) $ is an infinite discrete subset in $\mathcal{B}_{\mathcal{C}P^{n }, \mathcal{O}( kk_0 )}$. It follows that $ \Gamma_M \omega $ is an infinite discrete subset in $\mathcal{B}_{M,L^k}$. Now Lemma \ref{counterequivariantlmm} implies that for each $k\in\mathbb{N}$, $\mu_{Ca} ( \mathcal{B}_{M,L^k} ) = \mu_{Ma} ( \mathcal{B}_{M,L^k} ) =\infty $, which proves the theorem.
\end{proof}

Now we show that the ample line bundle $L$ on $M$ is always isometric to $ L_1 \otimes \pi_2^* \mathcal{O} (k_0) $ for some $k_0 \in\mathbb{N} $, where $\pi_1 : M\to M_1$ and $\pi_2 :M\to \mathbb{C}P^n $ are the projections. The following lemma is known, but we give the proof here for convenience. See also \cite[Lemme 11]{cosa1}.

\begin{lmm} \label{lemmadecomposition}
    Let $M_1 ,M_2 $ be K\"ahler manifolds. Assume that $H^1 (M_2 ,\mathbb{C}) = 0 $. Then the natural homomorphism
    $$\pi_1^* \otimes \pi_2^* : H^1 (M_1 ,\mathcal{O}^* ) \oplus H^1 (M_2 ,\mathcal{O}^* ) \to H^1 (M_1 \times M_2 ,\mathcal{O}^* ) $$
    is an isomorphism, where $\pi_i : M_1 \times M_2 \to M_i $ is the projection, $i=1,2$. In particular, if $L$ is an ample line bundle on $M_1\times \mathbb{C}P^n$, then $L$ is isometric to $ L_1 \otimes \pi_2^* \mathcal{O} (k_0) $ for some $k_0 \in\mathbb{N} $, where $\pi_1 : M\to M_1$ and $\pi_2 :M\to \mathbb{C}P^n $ are the projections.
\end{lmm}

\begin{proof}
   By considering the long exact sequence derived from the exponential sequence on $M_i $,
   \[\begin{tikzcd}
0 \arrow{r}{ } & \mathbb{Z}_{M_i} \arrow{r}{2\pi \sqrt{-1} } & \mathcal{O}_{M_i} \arrow{r}{\exp } & \mathcal{O}^*_{M_i} \arrow{r}{ } & 0 ,
\end{tikzcd}
\]
$i=1,2$, one can obtain the exact sequence
\[\begin{tikzcd}
H^1 ( M_i ,\mathcal{O}_{M_i} ) \arrow{r}{\exp } & H^1 (M_i , \mathcal{O}^*_{M_i} ) \arrow{r}{c_1 } & H^2 (M_i , \mathbb{Z} ) \cap H^2 (M_i ,\mathbb{C}) ,
\end{tikzcd}
\]
where $H^2 (M_i , \mathbb{Z} ) \cap H^1 (M_2 ,\mathbb{C}) $ is the image of the morphism $\mathfrak{i}_{M_i} : H^2 (M_i , \mathbb{Z} ) \to H^2 (M_i ,\mathbb{C}) $ given by the inclusion map $\mathbb{Z} \to \mathbb{C} $. Clearly, $\mathrm{ker} (\mathfrak{i}_{M_i} ) = H^2 (M_i , \mathbb{Z} )_{tor} $ is the torsion part of $H^2 (M_i , \mathbb{Z} ) $.

Then the projections $\pi_i : M_1 \times M_2 \to M_i $, $i=1,2$, yield the following commutative diagram
\[\begin{tikzcd}
\oplus_{i=1}^{2} H^1 ( M_i ,\mathcal{O}_{M_i} ) \arrow{r}{\exp} \arrow{d}{\varphi} & \oplus_{i=1}^{2} H^1 ( M_i ,\mathcal{O}^*_{M_i} )\arrow{r}{c_1} \arrow{d}{\pi_1^* \otimes \pi_2^*} & \oplus_{i=1}^{2} H^2 ( M_i ,\mathbb{Z} ) \cap H^2 (M_i ,\mathbb{C}) \arrow{d}{\psi} \\
 H^1 ( M_1 \times M_2 ,\mathcal{O}_{M_1 \times M_2} ) \arrow{r}{\exp} & H^1 ( M_1 \times M_2 ,\mathcal{O}^*_{M_1 \times M_2} ) \arrow{r}{c_1} & H^2 ( M_1 \times M_2 ,\mathbb{Z} ) \cap H^2 ( M_1 \times M_2 ,\mathbb{C} ) .
\end{tikzcd}
\]

By the K\"unneth theorem for $H^2 ( M_1 \times M_2 ,\mathbb{Z} ) $ \cite[Theorem 3B.6]{aha1}, one can see that $ \psi$ is an isomorphism. Note that $H^1 (M_2 ,\mathbb{Z}) = 0 $ and $H^0 (M_i ,\mathbb{Z}) \cong \mathbb{Z} $, $i=1,2$. Similarly, one can apply the K\"unneth theorem for $ H^1 ( M_1 \times M_2 ,\mathcal{O}_{M_1 \times M_2} ) \cong H^{0,1} ( M_1 \times M_2 ,\mathbb{C} ) $ \cite[p.103 ($\ast $)]{gh1} to show that $\varphi $ is also an isomorphism. Thus, the homomorphism
$$\pi_1^* \otimes \pi_2^* : H^1 (M_1 ,\mathcal{O}^* ) \oplus H^1 (M_2 ,\mathcal{O}^* ) \to H^1 (M_1 \times M_2 ,\mathcal{O}^* ) $$
is an isomorphism, which precisely concludes the lemma.
\end{proof}

Now we return to the proof of Theorem {\ref{thmexampleCPN}}.

\vspace{0.2cm}

\noindent \textbf{Proof of Theorem {\ref{thmexampleCPN}}:}
Let $M_1$ be a projective manifold (which can be a single point), $M=M_1 \times \mathbb{C}P^n $ for some $n\in\mathbb{N}$, and $L $ be a very ample line bundle on $M$. Then Lemma \ref{lemmadecomposition} implies that $ L \cong \pi^* L_1 \otimes \pi_2^* \mathcal{O} (k_0) $ for some $k_0 \in\mathbb{N} $, where $\pi_1 : M\to M_1$ and $\pi_2 :M\to \mathbb{C}P^n $ are the projections. Hence one can apply Theorem \ref{thmexampleCPNproductL} to show that the Calabi volume $\mu_{Ca} ( \mathcal{B}_{M,L^k} ) $ and the Mabuchi volume $\mu_{Ma} ( \mathcal{B}_{M,L^k} ) $ are infinite for each $k$. This completes the proof of Theorem {\ref{thmexampleCPN}}. \qed


 \end{document}